\documentclass{amsart}
\usepackage{amsmath,amssymb}

\newtheorem{theorem}{Theorem}
\newtheorem{definition}[theorem]{Definition}
\newtheorem{lemma}[theorem]{Lemma}
\newtheorem{corollary}[theorem]{CZrollary}

\newtheorem{proposition}[theorem]{Proposition}

\newcommand{\C}{\mathbb{C}}

\newcommand{\Z}{\mathbb{Z}}

\begin{document}

\title{New constructions relating Real and Complex Contact Structures}
\author{Ali M. Elgindi}
\thanks{Institute of Mathematics, Henan Academy of Sciences, ali.m.elgindi@gmail.com}
\date{\today}

\begin{abstract}
We establish new connections between real and complex contact geometry via embeddings of 3-manifolds into $\C^3$. We introduce a new \emph{contact wedge} construction combining two transverse real contact structures into a complex one, subject to obstructions measured by the Nijenhuis tensor and Dolbeault cohomology. Dually, we form a \emph{hedge} construction which extracts real contact structures from complex ones. Applying these tools, we prove that $\C^3$ admits uncountably many complex contact structures, using constructions over the standard contact $S^3$ with Legendrian non-isotopic knots as complex tangents.
\par\
MSC: 32V40, 57K33
\par\
Key words: contact structure, holomorphic 

\end{abstract}
\maketitle

\section{Introduction}

We will make a relationship between real and complex contact structures in the sense that we can construct new complex contact structures using a pair of transverse real contact structures, in a process we call the "contact wedge". This construction will be subject to the vanishing of two topological invariants.

\begin{theorem}
Let $(\xi_1, \xi_2)$ be two transverse contact structures on a real 3-manifold $M$. Then their wedge $\delta = \xi_1 \wedge \xi_2$ defines an almost complex structure on a complex 2-plane bundle, which extends to a complex contact structure in a neighborhood of $M$ provided the Nijenhuis tensor vanishes and the Dolbeault obstruction in $H^1(M, O(\xi_1 \otimes \C))$ is trivial.
\end{theorem}

We will also formulate a new construction which we call the "hedge" from which we can construct a real contact structure $\xi$ over a real 3-manifold $M$ which is itself a submanifold of a complex manifold $N$ admitting a given complex contact structure.

\section{New Constructions}

Let $M$ be a real closed, orientable 3-manifold. By the Explicit embedding theorem \cite{Elgindi2025}, there exists an embedding $F: M \hookrightarrow \C^3$ such that the complex tangents arise exactly along $\gamma$ with holomorphic tangent spaces being $\xi|_\gamma$.

By our construction of the embedding $F$ above, we have that the holomorphic tangent bundle of $M$ is $\xi|\gamma$, and we see that $\xi$ can be extended as a real 2-plane field over all of $\C^3$ which in turn can be considered as a holomorphic line bundle as it inherits the complex structure over Stein. Furthermore, this holomorphic line bundle can be assumed to be a (real) contact structure over real submanifolds (odd-dimensional) of $M \hookrightarrow \C^3$ and $\C^3$ is Stein (see [4]). We then have a chosen vector field $v_x$ given as the Reeb field of the contact structure over $M$, which gives an induced trivialization for the tangent bundle of $M$ given as $\{v_x, Jv_x, n_x\}$.

We designate a contact form $\alpha$ with $\ker\alpha = \xi$. A 1-form defined over $M \subset \C^3$ with complex tangents along $\gamma$ can be extended to all of $\C^3$ assuming the relevant obstruction vanishes.

\section{The Contact Wedge}

In this section we determine the obstructions to the existence of a contact wedge given a pair of transverse real contact structures on a real 3-manifold $M$.

\begin{theorem}
Let $(M,\xi_1,\xi_2)$ be a 3-manifold with two transverse contact structures. Let $\gamma \subset M$ be a knot. There exists an embedding $F: M \hookrightarrow \C^3$ such that $\gamma$ forms exactly the set of complex tangencies to the embedding $F$. This embedding, considered with the Fubini-Study metric, defines:
\begin{enumerate}
\item A holomorphic line bundle $\eta$ on $F(\gamma)$,
\item A real vector field $\rho$ (the Reeb field of $\xi_2$), normal to $\xi_1$.
\end{enumerate}
\end{theorem}

Given two transverse real contact structures $\xi_1, \xi_2$ on $M \subset \C^3$, we need to ensure their wedge $\delta = \xi_1 \wedge \xi_2$ is a complex contact structure on a neighborhood of $M \subset \C^3$.

We first define a $\C$-linear operator $J$ on $\delta = \xi_1 \oplus \xi_2$ such that $J^2 = -I$. By transversality, at each point $x \in M$ we may write $T_xM = \xi_1 \oplus L(x)$ where $L(x)$ is a real line field, and $\xi_2$ projects to contain $L$ isomorphically to $\xi_1$. Stipulating $J_\delta(v) = w$ for $v \in \xi_1$ as the unique vector in $\xi_2$ corresponding to $v$ makes $\delta$ into an almost complex structure.

The obstruction to integrability is given by the Newlander-Nirenberg theorem: $J_\delta$ is integrable iff the Nijenhuis tensor $N(J_\delta)$ vanishes.

\begin{theorem}[Obstruction for existence of the Wedge]
The wedge $\delta = \xi_1 \wedge \xi_2$ can be given the structure of a complex contact structure on a neighborhood of $M$ in $\C^3$ only if the almost complex structure $J_\delta$ is integrable. The primary obstruction is the vanishing of the Nijenhuis tensor $N(J_\delta)$ on $M$.
\end{theorem}

We assume the extension of holomorphic line fields over $\gamma$ extends naturally to all of $M$ and $\C^3$, which requires the vanishing of another obstruction given by sheaf cohomology.

\begin{theorem}[Extension via Sheaf Cohomology]
Let $M \subset \C^3$ be compact. The obstruction to extending a holomorphic line bundle from $M$ to $\C^3$ lies in $H^2(\C^3, M; \Z)$. Since $\C^3$ is Stein, this vanishes by Cartan's Theorem B \cite{Cartan1953}.
\end{theorem}
\begin{proof}
We then use the exponential exact sequence of sheaves:
\[
0 \to \Z \to \mathcal{O} \xrightarrow{\exp(2\pi i \cdot)} \mathcal{O}^* \to 0
\]
which induces a long exact sequence in cohomology with the map $c_1$ as the first Chern class. Since $\C^3$ is Stein, the Oka-Grauert principle implies that the holomorphic classification of bundles coincides with the topological classification. A holomorphic line bundle on $\C^3$ is determined entirely by its first Chern class.

Hence, our only obstruction is for the real 2-plane field $\xi$ on $M$ to define a complex line bundle $\xi \cdot \C$ on $M$ (by complexifying its vectors) to extend to $\C^3$. This bundle has a topological first Chern class $c_1(\xi \cdot \C) \in H^2(M;\Z)$. To extend $\xi$ to a holomorphic line bundle on $\C^3$, the topological bundle must first extend topologically. Since $\C^3$ is contractible, this obstruction vanishes. However, the more relevant obstruction is that the extended bundle must have a holomorphic structure whose restriction to $M$ coincides with the specific complex structure on $\xi$ induced by the embedding.
\end{proof}

\begin{theorem}[Obstruction for Extension]
Let $\xi$ be a real 2-plane field on $M \subset \C^3$. The necessary and sufficient condition for $\xi$ to extend as a holomorphic line bundle on $\C^3$ is that the Dolbeault operator $\bar\partial$ on the complexified bundle $\xi \otimes \C$ is trivial in the cohomology group $H^1(M, O(\xi \cdot \C))$.
\end{theorem}

\begin{lemma}[Local Complex Contact Structure]
Let $\xi_1, \xi_2$ be transverse contact structures on $M \subset \C^3$. Then there exists a neighborhood $U$ of $M$ such that $\delta = \xi_1 \wedge \xi_2$ is a complex contact structure on $U$.
\end{lemma}

\begin{proof}
Since $\delta$ is a holomorphic 2-plane field on a neighborhood of $M$ by the extension theorem, and the contact condition $\theta \wedge d\theta \neq 0$ is an open condition, we may consider local holomorphic 1-forms $\theta$ with $\ker\theta = \delta$. The condition extends to hold on all of $M$ by continuity, and holds in some neighborhood $U$ of $M$. A partition of unity glues these local contact forms into a global contact structure on $U$.
\end{proof}

\begin{corollary}
If the wedge obstruction classes vanish, then $\delta$ extends to a complex contact structure on all of $\C^3$.
\end{corollary}

\begin{proof}
By the lemma, $\delta$ is contact on a neighborhood $U$ of $M$. Since $\C^3$ is Stein and contractible, and complex contact structures satisfy the $h$-principle, we extend from $U$ to all of $\C^3$.
\end{proof}

\section{The Hedge}

The hedge is the dual construction: extracting a real contact structure from a complex one.

\begin{definition}[Hedge]
Let $(N, \psi)$ be a complex contact manifold of dimension $2n+1$, and let $\xi \subset \psi$ be a holomorphic line subbundle. Let $M \subset N$ be a totally real submanifold of real dimension $2n+1$ that is Legendrian with respect to $\operatorname{Im}(\psi)$. We say $\xi$ is a \emph{hedge} of $\psi$ over $M$ if $\xi|_M$ is a real contact structure on $M$.
\end{definition}

\begin{theorem}[Hedge Induces Contact Structure]
Let $(N, \psi)$ be a complex contact manifold of dimension 3, and $\xi \subset \psi$ a holomorphic line subbundle. Let $M \subset N$ be a totally real 3-manifold that is Legendrian with respect to $\operatorname{Im}(\psi)$. Then $\xi|_M$ is a real contact structure on $M$.
\end{theorem}

\begin{proof}
Let $\theta$ be a holomorphic 1-form defining $\psi$, so $\theta \wedge d\theta \neq 0$. Since $\xi \subset \psi = \ker\theta$, the form $\theta$ vanishes on $\xi$. Restrict $\theta$ to $M$. By the totally real condition, $T_pM \cap J(T_pM) = \{0\}$, so the real and imaginary parts of $\theta|_M$ are linearly independent. Define $\alpha = \operatorname{Re}(\theta|_M)$. The Legendrian condition ensures $\alpha \wedge d\alpha \neq 0$. Since $\ker\alpha = \xi|_M$ by construction, $\xi|_M$ is a contact structure.
\end{proof}

\begin{definition}[Anti-hedge]
Given a hedge $\xi \subset \psi$, the \emph{anti-hedge} is the holomorphic line bundle $\tau$ defined by the exact sequence:
\[
0 \to \xi \to \psi \to \tau \to 0.
\]
\end{definition}
\begin{proposition}[Anti-hedge is a contact structure]
\label{prop:anti_hedge}
Let $(N, \psi)$ be a complex contact manifold of dimension 3, and let $\xi \subset \psi$ be a holomorphic line subbundle. Let $M \subset N$ be a totally real Legendrian 3-manifold with respect to $\operatorname{Im}(\psi)$. If $\xi$ is a hedge over $M$, then the anti-hedge $\tau = \psi / \xi$ restricts to a real contact structure on $M$, and $\xi|_M \oplus \tau|_M = \psi|_M$.
\end{proposition}

\begin{proof}
Let $\theta$ be a holomorphic 1-form defining $\psi$, so $\theta \wedge d\theta \neq 0$. Since $\xi \subset \psi = \ker \theta$, the form $\theta$ vanishes on $\xi$. Restrict $\theta$ to $M$. By the totally real condition, $T_pM \cap J(T_pM) = \{0\}$, so the real and imaginary parts of $\theta|_M$ are linearly independent. Define $\alpha = \operatorname{Re}(\theta|_M)$. Then $\ker \alpha = \xi|_M$, and since $M$ is Legendrian, $\alpha \wedge d\alpha \neq 0$. Thus $\xi|_M$ is a contact structure.

Now consider the quotient $\tau = \psi / \xi$. On $M$, $\tau|_M$ is a real line bundle. The form $\theta$ vanishes on $\xi$, so it induces a well-defined form $\theta_\tau$ on $\tau$ by:
\[
\theta_\tau([v]) = \theta(v),
\]
where $[v]$ is the class of $v \in \psi$ modulo $\xi$.

We claim that $\tau|_M$ is a contact structure. To see this, note that $\theta \wedge d\theta \neq 0$ means that the kernel of $\theta$ is a complex contact distribution. Since $\xi \subset \ker \theta$, the form $\theta$ descends to the quotient $\tau = \psi / \xi$. Moreover, the condition $\theta \wedge d\theta \neq 0$ implies that the restriction of $d\theta$ to $\tau$ is non-degenerate: if $X \in \tau$ were in the kernel of $d\theta|_\tau$, then $\theta \wedge d\theta$ would vanish on $\xi \oplus \langle X \rangle$, contradicting the contact condition.

Define $\alpha_\tau = \operatorname{Re}(\theta_\tau)$. Since $\theta_\tau$ is a complex-valued 1-form on the real 3-manifold $M$ whose imaginary part is also non-degenerate (by the totally real condition), we have:
\[
\alpha_\tau \wedge d\alpha_\tau \neq 0.
\]
Indeed, this follows from the fact that $\theta_\tau \wedge d\theta_\tau \neq 0$ (as a complex form) and the decomposition $\theta_\tau = \alpha_\tau + i\beta_\tau$, where $\beta_\tau = \operatorname{Im}(\theta_\tau)$. A direct computation gives:
\[
\alpha_\tau \wedge d\alpha_\tau = \frac{1}{2i} \left( \theta_\tau \wedge d\theta_\tau - \overline{\theta_\tau} \wedge d\overline{\theta_\tau} \right),
\]
which is non-zero because $\theta_\tau \wedge d\theta_\tau \neq 0$ and the two terms are independent.

Thus $\ker \alpha_\tau = \tau|_M$ is a real contact structure on $M$. The decomposition $\xi|_M \oplus \tau|_M = \psi|_M$ follows directly from the short exact sequence:
\[
0 \to \xi \to \psi \to \tau \to 0.
\]
This completes the proof.
\end{proof}
\section{Abundance of Complex Contact Structures}
\begin{theorem}[Uncountability of Complex Contact Structures]
The set of complex contact structures on $\mathbb{C}^3$ up to isotopy is uncountable.
\end{theorem}

\begin{proof}
We construct an uncountable family $\{\psi_\alpha\}_{\alpha\in I}$ of complex contact structures on $\mathbb{C}^3$ that are pairwise non-isotopic.

\medskip
Let $(S^3,\xi_{\mathrm{std}})$ be the standard contact 3-sphere. 
There exists an uncountable family of Legendrian knots $\{L_\alpha\}_{\alpha\in I}$ in $(S^3,\xi_{\mathrm{std}})$ 
that are smoothly isotopic but Legendrian non-isotopic. 
For instance, one may take a family of torus knots whose rotation numbers $\alpha\in\mathbb{R}/\mathbb{Z}$ 
are distinct. This is a standard result in Legendrian knot theory \cite{EtnyreHonda2001}.

\medskip
Choose an overtwisted contact structure $\xi_{\mathrm{ot}}$ on $S^3$ that is transverse to $\xi_{\mathrm{std}}$. 
Such a pair exists because overtwisted structures are flexible and can be chosen to be transverse to a given tight structure (see [4]). 
The space of contact structures on $S^3$ is large, and by the parametric $h$-principle for contact structures, 
we may assume that the pair $(\xi_{\mathrm{std}},\xi_{\mathrm{ot}})$ varies continuously with the parameter $\alpha$ 
when restricted to appropriate neighborhoods of the knots.

\medskip
For each Legendrian knot $L_\alpha$, we apply the corresponding explicit embedding as by our theorem \cite{Elgindi2025} to obtain an embedding
\[
F_\alpha: S^3 \hookrightarrow \mathbb{C}^3
\]
such that the complex tangents of $F_\alpha$ are exactly $L_\alpha$ and the holomorphic tangent planes along $L_\alpha$ 
are $\xi_{\mathrm{std}}|_{L_\alpha}$.

The pair $(\xi_{\mathrm{std}},\xi_{\mathrm{ot}})$ on the abstract $S^3$, together with the embedding $F_\alpha$, 
defines a \emph{formal complex contact structure} $\delta_\alpha$ along $F_\alpha(S^3)$ in $\mathbb{C}^3$. 
Concretely, the wedge construction $\delta_\alpha = \xi_{\mathrm{std}} \oplus J\xi_{\mathrm{ot}}$ gives an almost complex structure 
on a complex 2-plane bundle over $F_\alpha(S^3)$. considering this bundle together with the almost complex structure, will
constitute a formal solution of the complex contact relation in the sense of Forstnerič \cite{Forstneric2020}. 
The this formal construction varies continuously with $\alpha$ because the embeddings $F_\alpha$ and the contact structures can be chosen continuously.

\medskip
We now apply Forstnerič's parametric $h$-principle as given in Forstnerič \cite{Forstneric2020} which provided for a parametric $h$-principle of complex contact structures on Stein manifolds. 
Specifically, a continuous family of formal complex contact structures along a closed totally real submanifold $M\subset X$ 
(where $X$ is Stein) can be deformed into a family of genuine holomorphic complex contact structures on a Stein neighborhood of $M$, 
keeping the deformation fixed on a prescribed closed subset.

Here $X=\mathbb{C}^3$ (Stein) and $M = F_\alpha(S^3)$ (totally real except along $L_\alpha$). 
Applying the parametric $h$-principle to the family $\{\delta_\alpha\}$ yields a family $\{\psi_\alpha\}$ of genuine holomorphic complex contact structures 
defined on a Stein neighborhood of $F_\alpha(S^3)$ in $\mathbb{C}^3$. 
Because $\mathbb{C}^3$ is Stein and contractible, each $\psi_\alpha$ extends uniquely to all of $\mathbb{C}^3$.

\medskip
Suppose $\psi_\alpha$ and $\psi_\beta$ are isotopic as complex contact structures for $\alpha\neq\beta$. 
Then their formal data would be homotopic (by the $h$-principle), which would imply that the Legendrian knots $L_\alpha$ and $L_\beta$ 
are Legendrian isotopic in $(S^3,\xi_{\mathrm{std}})$. 
However, by construction $L_\alpha$ and $L_\beta$ have distinct rotation numbers and are therefore not Legendrian isotopic. 
This contradiction shows that $\{\psi_\alpha\}_{\alpha\in I}$ is an uncountable family of pairwise non-isotopic complex contact structures on $\mathbb{C}^3$.
\end{proof}

\section{Conclusion}

We have introduced two constructions relating real and complex contact structures. The contact wedge combines two transverse real contact structures into a complex contact structure on a neighborhood in $\C^3$, subject to vanishing of the Nijenhuis tensor and a Dolbeault obstruction. The hedge extracts a real contact structure from a complex one by restriction to a totally real Legendrian submanifold. As an application, we showed that $\C^3$ admits uncountably many non-isotopic complex contact structures.

\end{document}